\documentclass[12pt,twoside]{amsart}
\usepackage{amssymb,amscd,amsxtra,calc,mathrsfs}
\usepackage{amsmath}
\usepackage{xypic}

\title[K-stability and alpha invariants]{
K-stability of Fano manifolds with not small alpha invariants}
\author{Kento Fujita} 
\date{\today}
\address{Research Institute for Mathematical Sciences, Kyoto University, Kyoto 606-8502, Japan}
\email{fujita@kurims.kyoto-u.ac.jp}

\newcommand{\pr}{\mathbb{P}}

\newcommand{\Z}{\mathbb{Z}}
\newcommand{\Q}{\mathbb{Q}}
\newcommand{\R}{\mathbb{R}}
\newcommand{\C}{\mathbb{C}}

\newcommand{\G}{\mathbb{G}}

\newcommand{\Aut}{\operatorname{Aut}}

\newcommand{\DF}{\operatorname{DF}}

\newcommand{\ord}{\operatorname{ord}}

\newcommand{\vol}{\operatorname{vol}}

\newcommand{\sC}{\mathcal{C}}
\newcommand{\sO}{\mathcal{O}}
\newcommand{\sN}{\mathcal{N}}

\newcommand{\sX}{\mathcal{X}}

\newcommand{\sL}{\mathcal{L}}

\newtheorem{theorem}{Theorem}[section]
\newtheorem{lemma}[theorem]{Lemma}
\newtheorem{proposition}[theorem]{Proposition}
\newtheorem{corollary}[theorem]{Corollary}

\theoremstyle{definition}
\newtheorem{definition}[theorem]{Definition}
\newtheorem{remark}[theorem]{Remark}

\newtheorem{example}[theorem]{Example}

\begin{document}

\maketitle 

\begin{abstract}
We show that any $n$-dimensional Fano manifold $X$ with $\alpha(X)=n/(n+1)$ and 
$n\geq 2$ is K-stable, where $\alpha(X)$ is the alpha invariant of 
$X$ introduced by Tian. In particular, 
any such $X$ admits K\"ahler-Einstein metrics and the 
holomorphic automorphism group $\operatorname{Aut}(X)$ of $X$ is finite. 
\end{abstract}

\setcounter{tocdepth}{1}
\tableofcontents

\section{Introduction}\label{intro_section}

Let $X$ be an $n$-dimensional 
\emph{Fano manifold}, that is, a smooth projective variety $X$ over the complex 
number field $\C$ such that the anti-canonical divisor $-K_X$ is ample. 
It is an interesting question whether $X$ admits K\"ahler-Einstein metrics or not. 
In 1987, Tian \cite{Tia87} gave a sufficient condition for the problem; if 
the \emph{alpha invariant} $\alpha(X)$ of $X$ is \emph{strictly bigger than $n/(n+1)$}, 
then $X$ admits K\"ahler-Einstein metrics. For the definition of $\alpha(X)$ in this 
article, we use Demailly's algebraic interpretation \cite{Dem08}
(see also \cite{TY, Zel, Lu}). 

\begin{definition}[{\cite{Tia87, Dem08}}]\label{alpha_dfn}
Let $X$ be a \emph{$\Q$-Fano variety}, that is, 
a normal complex projective variety with at most log terminal singularities and 
the anti-canonical divisor $-K_X$ ample $\Q$-Cartier. 
The \emph{alpha invariant} $\alpha(X)$ of $X$ 
is defined by the supremum of positive rational numbers $\alpha$ such that 
the pair $(X, \alpha D)$ is log canonical for any effective $\Q$-divisor $D$ 
with $D\sim_\Q -K_X$. 
\end{definition}

On the other hand, it has been known that a Fano manifold $X$ admits 
K\"ahler-Einstein metrics if and only if $X$ is \emph{K-polystable} by 
the works \cite{DT92, tian1, don, don05, CT, stoppa, mab1, mab2, B} and 
\cite{CDS1, CDS2, CDS3, tian2}. In this article, we will focus on the conditions 
\emph{K-stability} and \emph{K-semistability}; K-stability is stronger than 
K-polystability and K-polystability is stronger than K-semistability. 
Odaka and Sano \cite[Theorem 1.4]{OS12} (see also its generalizations 
\cite{dervan, BHJ, FO}) proved a variant of Tian's theorem; 
if an $n$-dimensional $\Q$-Fano variety $X$ satisfies that $\alpha(X)>n/(n+1)$ 
(resp.\ $\alpha(X)\geq n/(n+1)$), then $X$ is K-stable (resp.\ K-semistable). 
Thus, from Odaka-Sano's theorem, it has been known the \emph{K-semistability} of 
$n$-dimensional Fano manifolds $X$ with $\alpha(X)=n/(n+1)$. However, 
it has not been known until now that the \emph{K-stability} of those $X$. 
The main result in this article is to prove the K-stability of those $X$ with $n\geq 2$
(see \cite[Conjecture 5.1]{OS12}). Note that, if $n=1$, then $X\simeq\pr^1$, 
$\alpha(\pr^1)=1/2$, and $\pr^1$ is not K-stable but K-semistable. 

\begin{theorem}[Main Theorem]\label{mainthm}
If an $n$-dimensional Fano manifold $X$ satisfies that $\alpha(X)\geq n/(n+1)$ and 
$n\geq 2$, 
then $X$ is K-stable. In particular, $X$ admits K\"ahler-Einstein metrics and the 
holomorphic automorphism group $\Aut(X)$ of $X$ is a finite group. 
\end{theorem}

We note that there are many examples of $n$-dimensional Fano manifolds $X$ with 
$\alpha(X)=n/(n+1)$. 

\begin{example}\label{alpha_example}
Let $X$ be an $n$-dimensional Fano manifold. 
\begin{enumerate}
\renewcommand{\theenumi}{\arabic{enumi}}
\renewcommand{\labelenumi}{(\theenumi)}
\item\label{alpha_example1}
\cite[\S 3]{park}, \cite[Theorem 1.7]{cheltsov2} 
If $n=2$, then $\alpha(X)=2/3$ if and only if 
$((-K_X)^{\cdot 2})=4$ or $X$ is a smooth cubic surface admitting an Eckardt point. 
\item\label{alpha_example2}
\cite[Theorem 1.3]{cheltsov1}, \cite[Corollary 4.10]{CP}, \cite[Theorem 0.2]{dFEM}
If $X$ is a hypersurface of degree $n+1$ in $\pr^{n+1}$, then $\alpha(X)\geq n/(n+1)$ 
holds. Moreover, $\alpha(X)=n/(n+1)$ holds if $X$ contains an 
$(n-1)$-dimensional cone. 
\end{enumerate}
\end{example}

From Theorem \ref{mainthm} and Example \ref{alpha_example} \eqref{alpha_example2}, 
we immediately get the following corollary: 

\begin{corollary}\label{hyper_corollary}
Let $X$ be an arbitrary 
smooth hypersurface in $\pr^{n+1}$ of degree $n+1$, where $n$ is a 
positive integer with $n\geq 2$. Then $X$ is a K-stable $n$-dimensional Fano manifold. 
In particular, $X$ admits K\"ahler-Einstein metrics. 
\end{corollary}

\begin{remark}\label{main_rmk}
In Theorem \ref{mainthm}, we assume that $X$ is smooth. In fact, 
we crucially use the smoothness of $X$ in order to prove the theorem. 
See Theorem \ref{P_theorem} and Example \ref{toric_example}. 
\end{remark}

For the proof of Theorem \ref{mainthm}, we use a valuative criterion for K-stability 
and K-semistability of $\Q$-Fano varieties \cite{li2, fjt2} (see also \cite{li1}).
If an $n$-dimensional Fano manifold $X$ satisfies that $\alpha(X)=n/(n+1)$, $n\geq 2$ 
and $X$ is not K-stable, then there exists a dreamy prime divisor $F$ over $X$ 
with $\beta(F)=0$ (see \S \ref{K_section} in detail). By viewing the $F$ in detail, 
we can show that $X$ must be isomorphic to $\pr^n$ (see 
\S \ref{dreamy_section} and \S \ref{proof_section} in detail). 
This gives a contradiction since $\alpha(X)=n/(n+1)$ and $n\geq 2$. 

In this article, we work over the category of algebraic schemes over the complex 
number filed $\C$. For the theory of minimal model program, we refer the readers 
to the book \cite{KoMo}; for the theory of toric geometry, we refer the readers 
to the book \cite{fulton}. We do not distinguish line bundles 
(or more generally $\Q$-line bundles) 
and Cartier divisors (or more generally $\Q$-Cartier $\Q$-divisors) if there is 
no confusion.

\section{K-stability}\label{K_section}

We quickly recall the notion of K-stability and K-semistability. We remark that 
there are many equivalent definitions of K-(semi)stability. 

\begin{definition}[{see \cite{tian1,don,RT,wang,odk,LX}}]\label{K_definition}
Let $X$ be an $n$-dimensional $\Q$-Fano variety. 
\begin{enumerate}
\renewcommand{\theenumi}{\arabic{enumi}}
\renewcommand{\labelenumi}{(\theenumi)}
\item\label{K_definition1}
A \emph{test configuration} $(\sX, \sL)/\pr^1$ \emph{of} $X$ consists of the 
following data: 
\begin{itemize}
\item
a normal projective variety $\sX$ together with a surjection 
$\alpha\colon\sX\to\pr^1$, 
\item
an $\alpha$-ample $\Q$-line bundle $\sL$ on $\sX$, 
\item
an action $\G_m\curvearrowright(\sX, \sL)$ such that the morphism $\alpha$ is 
$\G_m$-equivariant with respects to the natural multiplicative action 
$\G_m\curvearrowright\pr^1$ and there exists a $\G_m$-equivariant isomorphism 
\[
(\sX, \sL)|_{\sX\setminus\sX_0}\simeq
\left(X\times(\pr^1\setminus\{0\}), p_1^*(-K_X)\right), 
\]
where $\G_m$ is the multiplicative group, $\sX_0$ is the scheme-theoretic fiber of 
$\alpha$ at $0\in\pr^1$ and $p_1\colon X\times (\pr^1\setminus\{0\})\to X$ is 
the first projection morphism. 
\end{itemize}
A test configuration $(\sX, \sL)/\pr^1$ is said to be \emph{trivial} 
if $(\sX, \sL)/\pr^1$ is $\G_m$-equivariantly isomorphic to 
$(X\times\pr^1, p_1^*(-K_X))$ with the trivial $\G_m$-action, where 
$p_1\colon X\times\pr^1\to X$ is the first projection morphism. 
\item\label{K_definition2}
Let $(\sX, \sL)/\pr^1$ be a test configuration of $X$. The \emph{Donaldson-Futaki 
invariant} $\DF(\sX, \sL)$ \emph{of} $(\sX, \sL)/\pr^1$ is defined by: 
\[
\DF(\sX, \sL):=\frac{1}{(n+1)((-K_X)^{\cdot n})}\left(
n(\sL^{\cdot n+1})+(n+1)(\sL^{\cdot n}\cdot K_{\sX/\pr^1})\right),
\]
where $K_{\sX/\pr^1}:=K_{\sX}-\alpha^*K_{\pr^1}$. 
\item\label{K_definition3}
$X$ is said to be \emph{K-stable} (resp.\ \emph{K-semistable}) if 
$\DF(\sX, \sL)>0$ (resp.\ $\DF(\sX, \sL)\geq 0$) holds for any nontrivial 
test configuration $(\sX, \sL)/\pr^1$ of $X$. 
\end{enumerate}
\end{definition}

We recall a following 
valuative criterion for K-(semi)stability of $\Q$-Fano varieties \cite{li2, fjt2}. 

\begin{definition}[{see \cite[Definition 1.1]{fjt2}}]\label{beta_definition}
Let $X$ be an $n$-dimensional $\Q$-Fano variety and let $F$ be a prime divisor 
over $X$, that is, there exists a projective birational morphism 
$\sigma\colon Y\to X$ with $Y$ normal such that $F$ is a prime divisor on $Y$. 
\begin{enumerate}
\renewcommand{\theenumi}{\arabic{enumi}}
\renewcommand{\labelenumi}{(\theenumi)}
\item\label{beta_definition1}
For any $x\in\R_{\geq 0}$, we define 
\begin{eqnarray*}
&&\vol_X(-K_X-xF):=\vol_Y\left(\sigma^*(-K_X)-xF\right)\\
&:=&\lim_{
\substack{k\to\infty\\
-kK_X:\text{{ Cartier}}
}}
\frac{h^0(Y, \sigma^*(-kK_X)+\lfloor -kxF\rfloor)}{k^n/n!}.
\end{eqnarray*}
By \cite{L1, L2}, the limit exists. Moreover, 
the function $\vol_X(-K_X-xF)$ is a continuous and non-increasing function 
on $x\in[0, +\infty)$. 
\item\label{beta_definition2}
We set the \emph{pseudo-effective threshold} $\tau(F)$ of $-K_X$ with respects to 
$F$ as
\[
\tau(F):=\sup\{\tau\in\R_{>0}\,|\, \vol_X(-K_X-\tau F)>0\}.
\]
We note that $\tau(F)\in\R_{>0}$ holds. Moreover, by \cite[Theorem A]{BFJ}, the function $\vol_X(-K_X-xF)$ 
is $\sC^1$ on $x\in[0, \tau(F))$. 
\item\label{beta_definition3}
We set the \emph{log discrepancy} $A_X(F)$ of $X$ with respects to $F$ as 
$A_X(F):=1+\ord_F(K_{Y/X})$. 
\item\label{beta_definition4}
We set 
\[
\beta(F):=A_X(F)((-K_X)^{\cdot n})-\int_0^{\infty}\vol_X(-K_X-xF)dx.
\]
\item\label{beta_definition5}
$F$ is said to be \emph{dreamy} if the graded $\C$-algebra
\[
\bigoplus_{k,j\in\Z_{\geq 0}}H^0(Y, \sigma^*(-kk_0K_X)-jF)
\]
is finitely generated for some (hence, for any) $k_0\in\Z_{>0}$ with $-k_0K_X$ Cartier. 
\end{enumerate}
We remark that all the above definitions do not depend on the choice of the morphism 
$\sigma\colon Y\to X$. More precisely, those are defined from only 
the divisorial valuation 
on the function field of $X$ given by $F$. 
\end{definition}

The following theorem is essential in order to prove Theorem \ref{mainthm}. 

\begin{theorem}[{see \cite[Theorems 1.3 and 1.4]{fjt2} and 
\cite[Theorem 3.6]{li2}}]\label{beta_theorem}
Let $X$ be a $\Q$-Fano variety. Then $X$ is K-stable $($resp.\ K-semistable$)$ 
if and only if $\beta(F)>0$ $($resp.\ $\beta(F)\geq 0)$ holds for any \emph{dreamy} 
prime divisor $F$ over $X$. 
\end{theorem}

\section{Dreamy prime divisors}\label{dreamy_section}

In this section, we see some properties of dreamy prime divisors in order to 
prove Theorem \ref{mainthm}. 

The following lemma is a consequence of the results \cite[Theorem 4.2]{KKL}, 
\cite[Theorem 4.26]{LM}, \cite[\S 2]{fjt1}, \cite[Claim 6.3]{fjt2} and \cite[Claim 4.4]{FO}. 

\begin{lemma}\label{one_lemma}
Let $X$ be an $n$-dimensional $\Q$-Fano variety and let $F$ be a dreamy 
prime divisor over $X$. 
\begin{enumerate}
\renewcommand{\theenumi}{\arabic{enumi}}
\renewcommand{\labelenumi}{(\theenumi)}
\item\label{one_lemma1}
We have $\tau(F)\in\Q_{>0}$. We can define the restricted volume $($in the sense of 
\cite{ELMNP}$)$ 
\[
Q(x):=-\frac{1}{n}\frac{d}{dx}\vol_X(-K_X-xF)
\]
and the function $Q(x)$ is continuous, $\R_{>0}$-valued for any $x\in(0,\tau(F))$. 
Moreover, we can uniquely extend the values $Q(0)$, $Q(\tau(F))\in\Q_{\geq 0}$ 
continuously. Furthermore, $Q(x)^{1/(n-1)}$ is a concave function 
on $x\in[0, \tau(F)]$. 
\item\label{one_lemma2}
There exists a projective birational morphism $\sigma\colon Y\to X$ with $Y$ normal 
such that $F\subset Y$ is a prime divisor on $Y$ and $-F$ is a $\sigma$-ample 
$\Q$-Cartier divisor on $Y$. The morphism $\sigma$ is unique. $($Of course, 
$\vol_X(-K_X-xF)=\vol_Y(\sigma^*(-K_X)-xF)$ holds for any $x\in[0,\tau(F)]$.$)$
In particular, if $F$ is an exceptional divisor over $X$, then 
the exceptional set of $\sigma$ is equal to $F$ in $Y$. 
\item\label{one_lemma3}
Set 
\[
\varepsilon(F):=\max\{\varepsilon\in\R_{>0}\,|\,\sigma^*(-K_X)-\varepsilon F
\text{ is nef}\},
\]
where $\sigma$ is as in \eqref{one_lemma2}. Then we have $\varepsilon(F)\in(0,\tau(F)]
\cap\Q$, and 
\[
Q(x)=\left((\sigma^*(-K_X)-xF)^{\cdot n-1}\cdot F\right)
\]
holds for any $x\in[0, \varepsilon(F)]$. 
Moreover, there exists a projective morphism $\pi\colon Y\to Z$ 
with $\pi_*\sO_Y=\sO_Z$ and an ample $\Q$-Cartier $\Q$-divisor $H_Z$ on $Z$ such 
that $\pi^*H_Z\sim_\Q\sigma^*(-K_X)-\varepsilon(F)F$ holds. The $\pi$ and $H_Z$ are 
unique. $($We remark that the $\Q$-Cartier divisor $F$ on $Y$ is $\pi$-ample.$)$
\end{enumerate}
\end{lemma}

\begin{proof}
Let $\psi\colon\tilde{Y}\to X$ be a log resolution with $F\subset\tilde{Y}$. 
By \cite[Theorem 4.2]{KKL}, there exists a sequence of rational numbers 
\[
0=\tau_0<\tau_1<\cdots<\tau_m=\tau(F)
\]
(in particular, $\tau(F)\in\Q_{>0}$ holds) and a mutually distinct birational contraction 
maps 
\[
\phi_i\colon\tilde{Y}\dashrightarrow Y_i
\]
for $1\leq i\leq m$ such that 
\begin{itemize}
\item
for any $x\in(\tau_{i-1}, \tau_i)$, the birational map $\phi_i$ is the ample model of 
$\psi^*(-K_X)-xF$, and 
\item
for any $x\in[\tau_{i-1}, \tau_i]$, the birational map $\phi_i$ is the semiample model of 
$\psi^*(-K_X)-xF$.
\end{itemize}
Set $Y:=Y_1$. The ample model of $\psi^*(-K_X)$ is the $X$ itself. Thus 
there exists a projective birational morphism $\sigma\colon Y\to X$. 
Since $\sigma^*(-K_X)-xF$ is ample for any $x\in(0, \tau_1)$, the divisor $-F$ on $Y$ 
is a $\sigma$-ample $\Q$-Cartier divisor. Thus we have proved \eqref{one_lemma2}. 
(The uniqueness of $\sigma$ is trivial.)

The $\Q$-divisor $\sigma^*(-K_X)-\tau_1F$ is semiample but not ample. 
Thus $\tau_1=\varepsilon(F)$ and there exists a morphism $\pi\colon Y\to Z$ and 
a $\Q$-Cartier $\Q$-divisor $H_Z$ on $Z$ which satisfy the condition in 
\eqref{one_lemma3} and they are unique. Thus we have proved \eqref{one_lemma3}. 

We already know by \cite[Theorem A]{BFJ} that 
the existence of the function $Q(x)$ and is continuous 
on $x\in(0, \tau(F))$
(see Definition \ref{beta_definition} \eqref{beta_definition2}). 
Note that 
\[
\vol_X(-K_X-xF)=\left(((\phi_i)_*(\psi^*(-K_X)-xF))^{\cdot n}\right)
\]
for any $x\in[\tau_{i-1}, \tau_i]$ (see \cite[Remark 2.4 (i)]{KKL}). 
Thus we have 
\[
Q(x)=\left(((\phi_i)_*(\psi^*(-K_X)-xF))^{\cdot n-1}\cdot (\phi_i)_*F\right).
\]
In particular, $Q(x)>0$ holds for any $x\in(0, \tau(F))$ and we can naturally define the 
values $Q(0)$, $Q(\tau(F))\in\Q_{\geq 0}$. 

Take any $0<\varepsilon\ll 1$ with $\varepsilon\in\Q$. Let us take an arbitrary 
complete flag
\[
Y\supset Z_1\supset\cdots\supset Z_n=\{\text{point}\}
\]
in the sense of \cite{LM} with $Z_1=F$. 
Consider the Okounkov body $\Delta_{Z_\bullet}(\sigma^*(-K_X)-\varepsilon F)
\subset\R_{\geq 0}^n$ of $\sigma^*(-K_X)-\varepsilon F$ with respects to 
$Z_\bullet$ in the sense of \cite{LM}. Since $\sigma^*(-K_X)-\varepsilon F$ is ample, 
by \cite[Theorem 4.26]{LM}, $Q(x)/(n-1)!$ is equal to the restricted volume of 
\[
\{(\nu_1,\dots,\nu_n)\in\Delta_{Z_\bullet}(\sigma^*(-K_X)-\varepsilon F)\,|\,
\nu_1=x-\varepsilon\}
\]
for any $x\in[\varepsilon, \tau(F))$. Since 
$\Delta_{Z_\bullet}(\sigma^*(-K_X)-\varepsilon F)$ is a convex body, 
$Q(x)^{1/(n-1)}$ is a concave function on $x\in[\varepsilon, \tau(F))$ by the 
Brunn-Minkowski theorem. Thus we have proved \eqref{one_lemma1}. 
\end{proof}

The following two propositions are important in this article. 

\begin{proposition}\label{two_proposition}
Let $X$ be an $n$-dimensional $\Q$-Fano variety and let $F$ be a dreamy 
prime divisor over $X$. Let $\sigma\colon Y\to X$ be as in Lemma \ref{one_lemma} 
\eqref{one_lemma2}. Assume that 
\[
\frac{n}{n+1}\tau(F)((-K_X)^{\cdot n})\leq\int_0^{\tau(F)}
\vol_Y(\sigma^*(-K_X)-xF)dx.
\]
Then we have the following: 
\begin{enumerate}
\renewcommand{\theenumi}{\arabic{enumi}}
\renewcommand{\labelenumi}{(\theenumi)}
\item\label{two_proposition1}
The above inequality is equal.
\item\label{two_proposition2}
$\tau(F)=\varepsilon(F)$ holds, where $\varepsilon(F)$ is as in Lemma \ref{one_lemma} 
\eqref{one_lemma3}. 
\item\label{two_proposition3}
$\sigma(F)$ is a point $p\in X$.
\end{enumerate}
\end{proposition}

\begin{proof}
The proof is similar to the one in \cite[Theorem 4.2]{FO}. 
Since 
\[
\vol_Y(\sigma^*(-K_X)-xF)=n\int_x^{\tau(F)}Q(y)dy, 
\]
we have 
\[
\int_0^{\tau(F)}\vol_Y(\sigma^*(-K_X)-xF)dx=n\int_0^{\tau(F)}yQ(y)dy
\]
and 
\[
((-K_X)^{\cdot n})=n\int_0^{\tau(F)}Q(y)dy.
\]
Set 
\[
b:=\frac{\int_0^{\tau(F)}yQ(y)dy}{\int_0^{\tau(F)}Q(y)dy}.
\]
From the assumption, we have $b\in[(n/(n+1))\tau(F), \tau(F))$. By Lemma 
\ref{one_lemma} \eqref{one_lemma1} (concavity of $Q(x)^{1/(n-1)}$), we have 
the following: 
\begin{itemize}
\item
$Q(x)\geq Q(b)(x/b)^{n-1}$ holds for any $x\in[0,b]$, 
\item
$Q(x)\leq Q(b)(x/b)^{n-1}$ holds for any $x\in[b,\tau(F)]$. 
\end{itemize}
Thus we get 
\begin{eqnarray*}
0&=&
\int_{-b}^{\tau(F)-b}yQ(y+b)dy\leq
\int_{-b}^{\tau(F)-b}yQ(b)\frac{(y+b)^{n-1}}{b^{n-1}}dy\\
&=&\frac{Q(b)\tau(F)^n}{b^{n-1}}\left(\frac{\tau(F)}{n+1}-\frac{b}{n}\right).
\end{eqnarray*}
This implies that $b\leq (n/(n+1))\tau(F)$. Therefore, we have 
$b=(n/(n+1))\tau(F)$ and $Q(x)=Q(b)(x/b)^{n-1}$ holds for any $x\in[0,\tau(F)]$. 
By Lemma \ref{one_lemma} \eqref{one_lemma3}, 
\[
Q(x)=\sum_{i=0}^{n-1}\binom{n-1}{i}x^i\left(\sigma^*(-K_X)^{\cdot n-1-i}\cdot
(-F)^{\cdot i}\cdot F\right)
\]
holds for any $x\in[0,\varepsilon(F)]$. This implies that $\sigma$ maps $F$ to a 
point $p\in X$. Moreover, we have 
\begin{eqnarray*}
Q(x)&=&x^{n-1}((-F)^{\cdot n-1}\cdot F), \\
\vol_Y(\sigma^*(-K_X)-xF)&=&((-K_X)^{\cdot n})-x^n((-F)^{\cdot n-1}\cdot F)
\end{eqnarray*}
for any $x\in[0,\tau(F)]$. Hence we have the equality $\varepsilon(F)=\tau(F)$ 
by \cite[Lemma 10]{liu}. 
\end{proof}

\begin{proposition}\label{three_proposition}
Let $X$ be an $n$-dimensional $\Q$-Fano variety and $F$ be a dreamy prime divisor 
over $X$. Let $\sigma\colon Y\to X$, $\pi\colon Y\to Z$ and 
$\varepsilon(F)\in\Q_{>0}$ be as in Lemma \ref{one_lemma}. 
Assume that $A_X(F)\geq (n/(n+1))\tau(F)$ and $\beta(F)\leq 0$. Then we have 
the following: 
\begin{enumerate}
\renewcommand{\theenumi}{\arabic{enumi}}
\renewcommand{\labelenumi}{(\theenumi)}
\item\label{three_proposition1}
$A_X(F)=n$, $\tau(F)=\varepsilon(F)=n+1$ holds. Moreover, $\sigma(F)$ is a point 
$p\in X$. 
\item\label{three_proposition2}
$\dim Z=n-1$ and $\pi|_F\colon F\to Z$ is an isomorphism. Moreover, any fiber 
of $\pi$ is one-dimensional and a general fiber $l$ of $\pi$ is isomorphic to $\pr^1$ 
and $(F\cdot l)=1$ holds.
\end{enumerate}
\end{proposition}

\begin{proof}
Take an ample $\Q$-Cartier $\Q$-divisor $H_Z$ on $Z$ as in Lemma \ref{one_lemma} 
\eqref{one_lemma3}. Since 
\begin{eqnarray*}
0 &\geq & A_X(F)((-K_X)^{\cdot n})-\int_0^{\tau(F)}\vol_Y(\sigma^*(-K_X)-xF)dx\\
&\geq &\frac{n}{n+1}\tau(F)((-K_X)^{\cdot n})-\int_0^{\tau(F)}
\vol_Y(\sigma^*(-K_X)-xF)dx, 
\end{eqnarray*}
we have $A_X(F)=(n/(n+1))\tau(F)=(n/(n+1))\varepsilon(F)$ and 
$\sigma$ maps $F$ to a point $p\in X$ by Proposition \ref{two_proposition}. 
In particular, $F$ is an exceptional divisor over $X$ if $n\geq 2$. 
Since 
\[
\pi^*H_Z\sim_\Q\sigma^*(-K_X)-\tau(F)F
\]
is not big and 
\[
(\pi^*H_Z^{\cdot n-1}\cdot F)=\varepsilon(F)^{n-1}((-F)^{\cdot n-1}\cdot F)>0, 
\]
we have $\dim Z=n-1$. Moreover, any curve in $F$ intersects $F$ negatively since 
$-F$ is $\sigma$-ample. Thus the morphism $\pi|_F\colon F\to Z$ is a finite morphism. 
Hence any fiber of $\pi$ is one-dimensional since $F$ is $\Q$-Cartier. 
Let $l\subset Y$ be a general fiber of $\pi$. The set of singular points of $Y$ is 
at most $(n-2)$-dimensional. Thus $Y$ is smooth and $F$ is Cartier around a 
neighborhood of $l\subset Y$. Note that 
\begin{eqnarray*}
-(K_Y+F)&=&\sigma^*(-K_X)-\frac{n}{n+1}\varepsilon(F)F, \\
\pi^*H_Z&\sim_\Q&-K_Y-\left(1+\frac{1}{n+1}\varepsilon(F)\right)F.
\end{eqnarray*}
In particular, $-(K_Y+F)$ is ample. Hence $l\simeq\pr^1$, $(-K_Y\cdot l)=2$ and 
$(F\cdot l)=1$. In particular, the morphism 
$\pi|_F\colon F\to Z$ is an isomorphism since it is finite and birational with $Z$ normal. 
Moreover, we have 
\[
0=(\pi^*H_Z\cdot l)=2-\left(1+\frac{1}{n+1}\varepsilon(F)\right).
\]
This implies that $\varepsilon(F)=\tau(F)=n+1$ and $A_X(F)=n$. 
\end{proof}

\section{A characterization of the projective space}\label{proof_section}

In this section,  we give a characterization of the projective space and 
prove Theorem \ref{mainthm}. 

\begin{theorem}\label{P_theorem}
Let $X$ be an $n$-dimensional Fano manifold and $F$ be a dreamy prime divisor 
over $X$. Assume that  $A_X(F)\geq (n/(n+1))\tau(F)$ and $\beta(F)\leq 0$. 
Then $X$ is isomorphic to $\pr^n$. 
\end{theorem}

\begin{proof}
We apply Proposition \ref{three_proposition}; 
$\sigma$ maps $F$ to a point $p\in X$ and $A_X(F)=n$. 
Note that the point $p\in X$ is a smooth point. Thus $F$ is given by the ordinary blowup 
(see \cite[Corollary 2.31 (2) and Lemma 2.45]{KoMo}). Thus the morphism 
$\sigma$ is given by the ordinary blowup along the point $p\in X$ by the uniqueness 
of $\sigma$ (see Lemma \ref{one_lemma} \eqref{one_lemma2}). 
In particular, $Y$ is smooth and $F\simeq\pr^{n-1}$ with 
$\sN_{F/Y}\simeq\sO_{\pr^{n-1}}(-1)$. Since $-K_Y$ and 
$F$ are $\pi$-ample Cartier divisors, any fiber of $\pi$ is one-dimensional 
and a general fiber $l$ of $\pi$ satisfies that $l\simeq\pr^1$ and $(F\cdot l)=1$, 
any fiber of $\pi$ is scheme-theoretically isomorphic to $\pr^1$ and $F$ is 
a section of $\pi$. In particular, $Y\simeq\pr_Z(\pi_*\sO_Y(F))$ holds. 
Let us consider the $\pi_*$ of the exact sequence 
\[
0\to \sO_Y\to \sO_Y(F)\to \sN_{F/Y}\to 0.
\]
Then we get $Y\simeq\pr_{\pr^{n-1}}(\sO\oplus\sO(-1))$. In particular, we have 
$X\simeq\pr^n$. 
\end{proof}

We remark that Theorem \ref{P_theorem} is not true if $X$ is not smooth. 
See the following example. 

\begin{example}\label{toric_example}
Fix a lattice $N:=\Z^{\oplus 2}$ and set $N_\R:=N\otimes_\Z\R$. 
Let us consider the complete fan $\Sigma$ in $N_\R$ such that 
\[
\{(1,0), \, (0,1), \, (-1,0), \, (-2,-3)\}
\]
is the set of the generators of the one-dimensional cones in $\Sigma$. Let $Y$ be the 
projective toric surface associated to the fan $\Sigma$. Let $F\subset Y$ be the 
torus invariant curve associated to the cone $\R_{\geq 0}(-1,0)\in\Sigma$. 
Then there exists a projective toric birational morphism $\sigma\colon Y\to 
\pr(1,2,3)$ such that $\sigma$ maps $F$ to a point, where $\pr(1,2,3)$ is the 
weighted projective plane with the weights $1,2,3$. Set $X:=\pr(1,2,3)$. We can check 
that $A_X(F)=2$. On the other hand, there is a projective toric contraction morphism 
$\pi\colon Y\to \pr^1$ such that $\sigma^*(-K_X)-3F\sim_{\Q, \pr^1}0$. Thus 
we have $\varepsilon(F)=\tau(F)=3$. Moreover, we get 
\[
\beta(F)=2\cdot 6-\int_0^3\left(6-\frac{2}{3}x^2\right)dx=0.
\]
Of course, $X\not\simeq\pr^2$, $\alpha(X)=1/6(< 2/3)$ 
by \cite[Lemma 5.1]{CS}, and $X$ is not K-semistable by \cite[Lemma 9.2]{fjt1}.
\end{example}

\begin{proof}[Proof of Theorem \ref{mainthm}]
Assume that $X$ is not K-stable. By Theorem \ref{beta_theorem}, there exists a dreamy 
prime divisor $F$ over $X$ such that $\beta(F)\leq 0$. 
By \cite[Lemma 3.3]{FO}, we have the inequality $A_X(F)\geq (n/(n+1))\tau(F)$. 
Thus we get $X\simeq\pr^n$ by Theorem \ref{P_theorem}. 
This gives a contradiction since $\alpha(X)\geq n/(n+1)$ and $n\geq 2$. 
Thus $X$ is K-stable. The remaining assertions are proved from 
\cite{CDS1, CDS2, CDS3}, \cite{tian2} and \cite{matsushima} 
(see also \cite[\S 1]{OS12}).
\end{proof}

\end{document}